\documentclass[11pt, reqno]{amsart}

\usepackage{amsmath,amssymb,amsthm,amsfonts,verbatim}
\usepackage{microtype}
\usepackage[all,2cell]{xy}
\usepackage{mathtools}
\usepackage{graphicx}
\usepackage{pinlabel}
\usepackage{hyperref}
\usepackage{mathrsfs}
\usepackage{color}
\usepackage[dvipsnames]{xcolor}
\usepackage{enumitem}
\usepackage{cite}
\usepackage{soul}

\CompileMatrices

\usepackage[top=1.2in,bottom=1.2in,left=1in,right=1in]{geometry}

\usepackage{hyperref} 
\hypersetup{          
	colorlinks=true, breaklinks, linkcolor=[RGB]{51 102 204}, filecolor=Orchid, urlcolor=[RGB]{51 102 204},
	citecolor=Orchid, linktoc=all, }
\usepackage[nameinlink]{cleveref}

\theoremstyle{plain}
\newtheorem{theorem}{Theorem}[section]
\newtheorem{maintheorem}{Theorem}

\newtheorem{maincor}[maintheorem]{Corollary}
    \AddToHook{env/maincor/begin}{\crefalias{maintheorem}{maincor}}
    \crefname{maincor}{corollary}{corollaries}

    \AddToHook{env/problem/begin}{\crefalias{theorem}{problem}}
    \crefname{problem}{problem}{problems}
\newtheorem{proposition}[theorem]{Proposition}
    \AddToHook{env/proposition/begin}{\crefalias{theorem}{proposition}}
    \crefname{proposition}{proposition}{propositions}
\newtheorem{lemma}[theorem]{Lemma}
    \AddToHook{env/lemma/begin}{\crefalias{theorem}{lemma}}
    \crefname{lemma}{lemma}{lemmas}

    \AddToHook{env/conjecture/begin}{\crefalias{theorem}{conjecture}}
    \crefname{conjecture}{conjecture}{conjectures}

    \AddToHook{env/fact/begin}{\crefalias{theorem}{fact}}
    \crefname{fact}{fact}{facts}
\newtheorem{corollary}[theorem]{Corollary}
    \AddToHook{env/corollary/begin}{\crefalias{theorem}{corollary}}
    \crefname{corollary}{corollary}{corollaries}

    \AddToHook{env/claim/begin}{\crefalias{theorem}{claim}}
    \crefname{claim}{claim}{claims}

\theoremstyle{definition}
\newtheorem{definition}[theorem]{Definition}
    \AddToHook{env/definition/begin}{\crefalias{theorem}{definition}}
    \crefname{definition}{definition}{definitions}

    \AddToHook{env/example/begin}{\crefalias{theorem}{example}}
    \crefname{example}{example}{examples}

    \AddToHook{env/construction/begin}{\crefalias{theorem}{construction}}
    \crefname{construction}{construction}{constructions}

    \AddToHook{env/exercise/begin}{\crefalias{theorem}{exercise}}
    \crefname{exercise}{exercise}{exercises}
\newtheorem{remark}[theorem]{Remark}
    \AddToHook{env/remark/begin}{\crefalias{theorem}{remark}}
    \crefname{remark}{remark}{remarks}
\newtheorem{question}[theorem]{Question}
    \AddToHook{env/question/begin}{\crefalias{theorem}{question}}
    \crefname{question}{question}{questions}

    \AddToHook{env/answer/begin}{\crefalias{theorem}{answer}}
    \crefname{answer}{answer}{answers}

    \AddToHook{env/convention/begin}{\crefalias{theorem}{convention}}
    \crefname{convention}{convention}{conventions}

\newcommand{\nc}{\newcommand}
\nc{\dmo}{\DeclareMathOperator}

\nc{\OO}{\mathcal{O}}
\nc{\cA}{\mathcal{A}}
\nc{\cB}{\mathcal{B}}
\nc{\sB}{\mathscr{B}}
\nc{\C}{\mathbb{C}}
\nc{\cC}{\mathcal{C}}
\nc{\sC}{\mathscr{C}}
\nc{\BB}{\mathbb{B}}
\nc{\LL}{\mathcal{L}}
\nc{\bd}{\mathbf{d}}
\nc{\DD}{\mathbb{D}}
\nc{\cD}{\mathcal{D}}
\nc{\sD}{\mathscr{D}}
\nc{\bF}{\mathbb{F}}
\nc{\cF}{\mathcal{F}}
\nc{\cG}{\mathcal{G}}
\nc{\cI}{\mathcal{I}}
\nc{\cK}{\mathcal{K}}
\nc{\cL}{\mathcal{L}}
\nc{\cM}{\mathcal{M}}
\nc{\bM}{\mathbf{M}}
\nc{\N}{\mathbb{N}}
\nc{\cN}{\mathcal{N}}
\nc{\cO}{\mathcal{O}}
\nc{\bp}{\mathbf{p}}
\nc{\bP}{\mathbb{P}}
\nc{\cP}{\mathcal{P}}
\nc{\Q}{\mathbb{Q}}
\nc{\R}{\mathbb{R}}
\nc{\cS}{\mathcal{S}}
\nc{\cT}{\mathcal T}
\nc{\cU}{\mathcal U}
\nc{\bV}{\mathbb V}
\nc{\cX}{\mathcal{X}}
\nc{\cY}{\mathcal{Y}}
\nc{\Z}{\mathbb{Z}}
\nc{\disk}{\mathbb{D}}
\nc{\hyp}{\mathbb{H}}

\nc{\CP}{\mathbb{CP}}
\nc{\RP}{\mathbb{RP}}
\dmo{\Mod}{Mod}
\dmo{\PMod}{PMod}
\dmo{\LMod}{LMod}
\dmo{\Diff}{Diff}
\dmo{\Homeo}{Homeo}
\dmo{\dist}{dist}
\dmo\BDiff{BDiff}
\dmo\SO{SO}
\dmo\Sp{Sp}
\dmo\Hom{Hom}
\dmo\SL{SL}
\dmo\rank{rank}
\dmo\sig{sig}
\dmo\Out{Out}
\dmo\Aut{Aut}
\dmo\Inn{Inn}
\dmo\GL{GL}
\dmo\PGL{PGL}
\dmo\Gr{Gr}
\dmo\PSL{PSL}
\dmo\BHomeo{BHomeo}
\dmo\EHomeo{EHomeo}
\dmo\EDiff{EDiff}
\dmo\Disc{Disc}
\dmo\Aff{Aff}

\renewcommand{\bar}{\overline}

\dmo\Teich{Teich}
\dmo\Fix{Fix}
\nc{\pair}[1]{\ensuremath{\left\langle #1 \right\rangle}}
\nc{\abs}[1]{\ensuremath{\left| #1 \right|}}
\nc{\action}{\circlearrowright}
\nc{\abcd}[4]{\ensuremath{\left(\begin{array}{cc} #1 & #2 \\ #3 & #4 \end{array}\right)}}
\nc{\into}\hookrightarrow
\dmo{\Isom}{Isom}
\nc{\normal}{\vartriangleleft}
\dmo{\Vol}{Vol}
\dmo{\im}{Im}
\dmo{\Push}{Push}
\dmo{\Conf}{Conf}
\dmo{\PConf}{PConf}
\dmo{\PB}{PB}
\dmo{\id}{id}
\dmo{\Jac}{Jac}
\dmo{\Pic}{Pic}
\dmo{\Stab}{Stab}
\dmo{\Arf}{Arf}
\dmo{\End}{End}
\dmo{\Ext}{Ext}
\dmo{\Gal}{Gal}
\dmo{\lcm}{lcm}
\dmo{\ab}{ab}
\dmo{\opp}{op}
\dmo{\SU}{SU}
\dmo{\OT}{\Omega \mathcal{T}}
\dmo{\OM}{\Omega \mathcal{M}}
\dmo{\PH}{\mathbb{P}\mathcal{H}}
\dmo{\spin}{spin}
\dmo{\even}{even}
\dmo{\odd}{odd}
\dmo{\comp}{\mathcal{H}}
\dmo{\Mgk}{\mathcal{M}_{g, \underline{\kappa}}}
\dmo{\orb}{orb}
\dmo{\AJ}{AJ}
\dmo{\Ck}{\mathsf{C}(\underline{\kappa})}
\dmo{\Int}{Int}
\dmo{\pr}{pr}
\dmo{\lab}{lab}
\dmo{\Sym}{Sym}
\dmo{\Ann}{Ann}
\dmo{\Rad}{Rad}
\dmo{\Ind}{Ind}
\dmo{\Div}{Div}
\dmo{\Res}{Res}
\dmo{\Hur}{Hur}
\dmo{\vcd}{vcd}
\dmo{\codim}{codim}

\nc{\Span}[1]{\operatorname{Span}(#1)}

\renewcommand{\epsilon}{\varepsilon}
\renewcommand{\tilde}{\widetilde}
\renewcommand{\le}{\leqslant}
\renewcommand{\ge}{\geqslant}
\nc{\coloneq}{\mathrel{\mathop:}\mkern-1.2mu=}
\nc{\margin}[1]{\marginpar{\scriptsize #1}}
\nc{\para}[1]{\medskip\noindent\textbf{#1.}}
\definecolor{myblue}{RGB}{102,153, 255}
\definecolor{myred}{RGB}{204,0,0}
\definecolor{mygreen}{RGB}{0,204,0}
\definecolor{myorange}{RGB}{255,102,0}
\definecolor{mypurple}{RGB}{138,43,226}
\nc{\red}[1]{\textcolor{myred}{#1}}
\nc{\blue}[1]{\textcolor{myblue}{#1}}

\newcommand{\maxdim}{3g-3}

\author[Kaufmann, Salter, Zhang, Zhong]{Julian Kaufmann, Nick Salter, Zhong Zhang, Xiyan Zhong}

\address{JK: Department of Mathematics, University of California, Berkeley, 970 Evans Hall, Berkeley, CA 94720}
\email{jkaufma2@alumni.nd.edu}

\address{NS: Department of Mathematics, University of Notre Dame, 255 Hurley Building, Notre Dame, IN 46556}
\email{nsalter@nd.edu}

\address{ZZ: Department of Mathematics, University of Chicago, 5734 S. University Ave., Chicago, IL 60637}
\email{zhongz@uchicago.edu}

\address{XZ: Max Planck Institute for Mathematics, Vivatsgasse 7, 53111 Bonn, Germany}
\email{xiyanmath@gmail.com}

\title[Linear reps of the mapping class group of dimension at most $3g-3$]{Linear representations of the mapping class group of dimension at most $3g-3$}
\date{May 28, 2026}

\begin{document}

\begin{abstract}
    We classify representations of the mapping class group of a surface of genus $g$ (with at most one puncture or boundary component) up to dimension $3g-3$. Any such representation is the direct sum of a representation in dimension $2g$ or $2g+1$ (given as the action on the (co)homology of the surface or its unit tangent bundle) with a trivial representation. As a corollary, any linear system on the moduli space of Riemann surfaces of genus $g$ in this range is of algebro-geometric origin.
\end{abstract}

\maketitle

\section{Introduction}
Let $S$ be an oriented surface of genus $g$, either closed, with one puncture, or with one boundary component (notated as $\Sigma_g, \Sigma_{g,*}, \Sigma_{g,1}$, respectively). Assume $g \ge 3$ throughout. Let $\Mod(S)$ denote the mapping class group of $S$.
This paper gives a classification of complex representations of $\Mod(S)$ in the dimension range $n \le \maxdim$. In dimension $2g$, there is the the {\em symplectic representation}
\[
\Psi: \Mod(S) \to \GL(H),
\]
where $H = H_1(S;\C)$ denotes the first homology of $S$, equipped with the intersection pairing $\pair{\cdot, \cdot}$. When $S$ is not closed, there is a representation of dimension $2g+1$
\[
\tilde \Psi: \Mod(S) \to \GL(\tilde H),
\]
where $\tilde H \cong \C^{2g+1}$ is a non-semisimple representation surjecting onto $H$ given by the action on the homology of the unit tangent bundle of $S$; see \Cref{subsection:KK}. There is also the non-isomorphic dual representation $\tilde H^*$.

Our main result shows that this gives a {\em complete} list of representations up to dimension $\maxdim$.

\begin{maintheorem}\label{theorem:main}
    Let $g \ge 3$, and let $\rho: \Mod(S) \to \GL(n, \C)$ be a nontrivial representation. Then for $n \le \maxdim$, $\rho$ is the direct sum of a trivial representation with one of the following:
    \[
    H, \tilde H, \mbox{ or the dual } \tilde H^*.
    \]
    If $S$ is closed, only $H$ can appear.
\end{maintheorem}

This extends the work of Korkmaz \cite{korkmaz}, who showed the uniqueness of the symplectic representation up to dimension $2g$, and Kasahara \cite{kasahara}, who classified representations of dimension $2g+1$.

\para{Idea of proof} Our analysis is centered around the notion of a {\em bi-affine representation}. A group representation $V$ is bi-affine if there is a filtration $V_1 \le V_2 \le V$ for which $V_1$ and $V/V_2$ are both trivial. Building off of the ideas developed by Korkmaz and Kasahara, we show inductively that every representation of $\Mod(S)$ up to dimension $3g-3$ is bi-affine. To do this, we develop a criterion for a $\Mod(S)$-representation to be bi-affine in \Cref{prop:whenDCS}. The other key results, \Cref{prop:unipotence,prop:codimenison,prop:kgtriv}, establish conditions under which the hypotheses of \Cref{prop:whenDCS} hold.
Separately, we show that any bi-affine representation of $\Mod(S)$ is in fact a direct sum of a trivial representation with a representation of dimension $2g+1$. This is a consequence of a homological calculation carried out in \Cref{prop:whencoboundary}. 

It is reasonable to ask why we use the language of bi-affine representations at all, if ultimately we are showing a stronger result. The answer is that we believe that this leads to the cleanest proofs of our results, with the least amount of fussing about cases and choosing coordinates. It is relatively painless to formulate a condition under which an extension of a bi-affine representation of $\Mod(S)$ remains bi-affine (cf. \Cref{lemma:againDCS}); the corresponding statement for an extension to be a direct sum with a trivial representation would be more elaborate, and would require us to carry around the data of the splitting for longer than necessary.

\para{Interpretation in terms of local systems on moduli spaces} 
Local systems on the moduli space $\mathcal{M}$ of Riemann surfaces are determined by monodromy representations
\[
\pi_1^{\mathrm{orb}}(\mathcal{M}) \to \GL(V),
\]
where $\pi_1^{\mathrm{orb}}(\mathcal{M})$ is the mapping class group. A local system $\bV$ on $\mathcal{M}$ is {\em of geometric origin} if there is a family $E \to \cM$ of smooth projective varieties over $\cM$ such that $\bV$ is a subquotient of the local system of (co)homology associated with $E$. It is conjectured that all semi-simple representations of mapping class groups are of geometric origin \cite{litt2024motivesmappingclassgroups}. 

Local systems of geometric origin are necessarily semi-simple, while the local systems associated with $\tilde H$ and $\tilde H^*$ are not. For the purposes of this discussion, we will say that a local system {\em arises algebro-geometrically} if it is the monodromy of a family as above, where the fibers of the morphism are now only required to be quasiprojective, but are still required to be a topological fibration. Theorem \ref{theorem:main} imposes strong constraints on the local systems that can appear on moduli spaces of closed surfaces or of surfaces with a puncture -- up to rank $3g-3$, they must arise algebro-geometrically.

\begin{maincor}
    For $g\geq 3$, any local system of rank at most $3g-3$ on $\cM_{g,1}$ or $\cM_{g}$ arises algebro-geometrically. Here are the families of algebraic varieties: 
    \begin{itemize}
    \item for the symplectic representation \( H \), it is the universal family of Riemann surfaces;
    \item for the representations \( \tilde{H} \) and \( \tilde{H}^* \), it is the relative tangent bundle of the universal family of curves with the zero section removed - the fiber of this bundle is homotopy equivalent to the unit tangent bundle of surfaces.
\end{itemize} 
\end{maincor}

\para{Applications to rigidity} In \cite{rigidity}, Farb proves that any nonconstant holomorphic map $f: \cM_{g,n} \to \cA_g$ must be the period mapping assigning a Riemann surface to its Jacobian (here, $\cA_g$ denotes the moduli space of principally polarized Abelian varieties of dimension $g$). The first step in the argument is to appeal to Korkmaz's work classifying representations of the mapping class group up to dimension $2g$, as this governs the possibilities for the induced map on orbifold fundamental groups. The work of this paper opens the way to extending Farb's work to give a classification of holomorphic maps $\cM_g \to \cA_h$ in the range $2h \le 3g-3$. We plan to revisit this topic in future work.

\para{Beyond Theorem \ref{theorem:main}} Here we offer some discussion and speculation about how the results of the paper may be extended. As a first comment, since this paper was first released, work of Brevidelli \cite{brevidelli} shows that any representation of $\Mod(S)$ of dimension at most $4g-4$ is non-faithful; Brevidelli moreover identifies an explicit subgroup of the Torelli group that must be in the kernel of any such representation. 

There are at least three axes along which the results of the paper could potentially be extended:
\begin{enumerate}
    \item Improving the dimension range in which all representations $\rho: \Mod(S) \to \GL(n,\C)$ are classified. In particular, identifying the next dimension $n$ in which an irreducible representation appears.
    \item Improving the result to handle general finite-type surfaces $S$ with an arbitrary number of boundary components and/or punctures.
    \item Improving the result to incorporate representations of (some class of) finite-index subgroups of $\Mod(S)$.
\end{enumerate}

As regards (1), the first question is the extent to which the range $n \le 3g-3$ in \Cref{theorem:main} is sharp. We do not know of any new representations appearing in dimension $n = 3g-2$. The bottleneck imposing the bound $n \le 3g-3$ of \Cref{theorem:main} is derived from \Cref{prop:codimenison}, which shows that in this range, any representation has Dehn twists acting as {\em transvections}.   \Cref{prop:unipotence,prop:whenDCS} hold in the larger range $n \le 4g-1$, resp. $n \le 4g-3$ while the remaining cornerstone result, \Cref{prop:kgtriv}, holds in any dimension.

There is still a large gulf between a conjectural improved range $n \le 4g-3$ in \Cref{theorem:main} and the dimensions of other known irreps of $\Mod(S)$. After the standard representation $\Psi$ acting on $H_1(S;\C)$, the next-smallest irreducible representation of $\Sp(2g;\Z)$ is $\wedge^2 H_1(S;\C)/\C$ of dimension $n = \binom{2g}{2}-1$, whose dimension grows {\em quadratically} in $g$. To the authors' knowledge, this (or more precisely, its pullback to $\Mod(S)$) is the next known irrep.

\begin{question}
   Does $\Mod(S)$ admit any irreducible representation of dimension $3g-2 \le n < \binom{2g}{2}-1$?
\end{question}

 The picture of irreducible representations of $\Mod(S)$ is even murkier when one restricts attention to those representations not factoring through the symplectic group. Constructions of such arise via the method of so-called {\em Prym representations}. The general idea here is to consider a finite covering $p:\tilde S \to S$, to which there is associated a finite-index subgroup $\Mod(S)[p] \le \Mod(S)$ consisting of mapping classes that lift to $\tilde S$. There is an action of $\Mod(S)[p]$ on $H_1(\tilde S;\Z)$ which can be inducted up to a representation of $\Mod(S)$ and subsequently decomposed into irreps. However, to the authors' knowledge, these constructions always have dimension {\em exponential} in the genus of $S$.

 \begin{question}
     Does $\Mod(S)$ admit any irreducible representation that does not factor through $\Sp(2g,\Z)$ whose dimension is polynomial in $g$?
 \end{question}

Regarding (2), there are some genuine novelties that arise when $S$ has more boundary components/punctures, say $p$ in total. Specifically, there is an action on $H_1(S;\Z)$, which is free of rank $2g+p-1$. This is non-semisimple, and indeed co-affine (cf. \Cref{def:biaff}), as the action on the subspace spanned by peripheral classes is trivial. It is reasonable to conjecture that in a certain range, every representation of $\Mod(S)$ continues to arise from an action on the (co)homology of $S$ or its unit tangent bundle; this seems like a reasonable target for future work.

Finally, with regards to (3), extending the classification to finite-index subgroups is in general a massive undertaking. For instance, the {\em Ivanov conjecture} posits that there is no finite-index subgroup of $\Mod(S)$ with nontrivial rational abelianization, i.e. that no finite-index subgroup admits a character with infinite image. While every finite-index subgroup contains some power of every Dehn twist, these will in general fragment into many distinct conjugacy classes, making their analysis vastly more complicated.
Moreover special relations holding between individual such twists evaporate upon passing to powers, and so the specific techniques (e.g. in \Cref{section:transvective}) used to obtain our results do not scale up. 

It may be more reasonable to ask for rigidity results for specific, well-understood subgroups. For instance, returning to the setting of Prym representations, let $x \in H^1(S; \Z/2\Z)$ be a nonzero vector, and define $\Mod(S)[x]$ as the stabilizer of $x$ under the action of $\Mod(S)$ on $H^1(S;\Z/2\Z)$. Such $x$ defines a $\Z/2\Z$-cover $\tilde S \to S$ (with $\tilde S$ consequently of genus $2g-1$), and there is a Prym representation $\rho: \Mod(S)[x] \to \Aut(H_1(\tilde S;\Z)) = \Sp(4g-2;\Z)$ obtained by acting on the homology of this cover (see \cite{looijengaprym} for a detailed analysis of the image of such maps). This representation splits into two summands via the eigenspace decomposition of $H_1(\tilde S;\Z)$ under the deck group action; the nontrivial eigenspace has rank $2g-2$, so there is an irrep of $\Mod(S)[x]$ of this dimension, in addition to the restriction of the symplectic rep. It was shown by Serv\'an \cite{servan} that this is the nontrivial representation of $\Mod(S)[x]$ of lowest dimension. How far can results of this nature be extended?

\para{Acknowledgments} We would like to thank Benson Farb, Nate Harman, Yasushi Kasahara, Aaron Landesman, Daniel Litt, and Andy Putman for their input. We are also grateful to an anonymous referee for their close reading and many helpful suggestions that led to a substantial improvement in the exposition. The second-named author is supported by NSF grant DMS-2338485. 

\section{A recollection of prior results}
\subsection{The work of Korkmaz and Kasahara}\label{subsection:KK}
The present work is deeply indebted to the papers \cite{korkmaz,kasahara} of Korkmaz and Kasahara. We recall their main results here.

\begin{theorem}[Korkmaz, Theorem 1 of \cite{korkmaz}]\label{thm:korkmaz2g-1}
For $g \ge 3$ and $n \le 2g -1$, any homomorphism $\rho: \Mod(S) \to \GL(n,\C)$ is trivial.
\end{theorem}

\begin{theorem}[Korkmaz, Theorem 2 of \cite{korkmaz}]\label{thm:korkmaz2g}
    For $g \ge 3$, any homomorphism $\rho: \Mod(S) \to \GL(2g, \C)$ is either trivial or else is conjugate to the symplectic representation $\Psi: \Mod(S) \to \GL(H)$.
\end{theorem}

We now turn to representations of dimension $2g+1$. Let $UTS$ denote the unit tangent bundle of $S$. When $S$ is non-closed, $UTS \cong S \times S^1$ splits as a product, and in particular, $H_1(UTS; \C) \cong \C^{2g+1}$. Let $\tilde H = H_1(UTS;\C)$. Any diffeomorphism of $S$ induces a diffeomorphism of $UTS$; the {\em unit tangent representation}
\[
\tilde \Psi: \Mod(S) \to \GL(\tilde H)
\]
is the induced action on homology. Note that $H \cong H_1(S;\C)$ arises as a quotient of $\tilde H$ via pushforward along the projection $\pi: UTS \to S$; the induced action on $H$ is evidently $\Psi$. This representation was studied by Trapp \cite{trapp}. Kasahara classified representations of $\Mod(S)$ for $S$ a surface of genus $g$ with an arbitrary number of punctures and/or boundary components. In the setting we consider here (where $S$ has at most one boundary component or puncture), his result specializes as follows.

\begin{theorem}[Kasahara, cf. Theorem 1.1 of \cite{kasahara}]\label{thm:kas}
    Let $S$ be a surface of genus $g \ge 7$ which is either closed, has one boundary component, or one puncture. If $S$ is closed, then every nontrivial representation $\rho: \Mod(S) \to \GL(2g+1,\C)$ is conjugate to $H \oplus \C$. If $S$ is nonclosed, then any nontrivial $\rho$ is conjugate to $H \oplus \C$, or to $\tilde H$, or to $\tilde H^*$.
\end{theorem}

\begin{remark}
    We explain here how to deduce this statement from \cite[Theorem 1.1]{kasahara}. The statement there asserts that isomorphism classes of nontrivial representations of dimension $2g+1$ up to dualizing are in bijection with $H^1(\Mod(S);H)/\C^\times$. When $S$ is closed, Morita computed $H^1(\Mod(S);H) = 0$, and when $S$ has one boundary component or puncture, $H^1(\Mod(S);H)\cong \C$ (see \cite{moritajacobians}, recalled here as \Cref{lemma:h1modH}). Thus up to dualizing, there is a unique representation not of the form $H\oplus \C$; it is not hard to show that $\tilde H$ is such a representation. For further discussion of $\tilde H$, see \Cref{subsection:UT}.
\end{remark}

\begin{remark}
    Our proof depends on \Cref{thm:korkmaz2g-1,thm:korkmaz2g} of Korkmaz, as well as various technical lemmas established therein. It is logically independent of Kasahara's main theorem (\Cref{thm:kas}), and indeed improves the range of his result from $g \ge 7$ down to $g \ge 4$, but we again make use of some of the internal technology.
\end{remark}

\subsection{The Torelli group and the work of Johnson}\label{subsection:johnson}
Recall that the {\em Torelli group} $\cI(S) \le \Mod(S)$ is the kernel of the symplectic representation $\Psi$. The structure of $\cI(S)$ was greatly clarified in a series of papers of Dennis Johnson in the 1980's. We recall the relevant portions of his theory here. 

We first recall the {\em Johnson homomorphism}, as developed in \cite{johnsonhom}. For the sake of expediency we will not describe this in full detail. For our purposes it is sufficient to know that the Johnson homomorphism is a map
\[
\tau: \cI(S) \to A_S,
\]
where $A_S$ is a certain finitely generated torsion-free abelian group. Define the {\em Johnson kernel}
\[
\cK(S) = \ker(\tau)
\]
as the kernel of $\tau$. 

A deep theorem of Johnson shows that $\cK(S)$ is generated by just two types of elements. Let $c \subset S$ be a separating curve (i.e. one for which $S \setminus c$ is disconnected). In the case where $S$ has a puncture or boundary component, the {\em genus} of $S$ is the genus of the subsurface {\em not} containing this; if $S$ is closed we define the genus as the smaller of the genera of the subsurfaces bounded by $c$. A {\em separating twist} is the Dehn twist $T_c$ about a separating curve $c$; we define the genus of such $T_c$ as the genus of $c$.

\begin{theorem}[Johnson, \cite{johnsonII}]\label{theorem:johnsonII}
    For $g \ge 3$, $\cK(S)$ is generated by the set of separating twists of genus one and two. 
\end{theorem}

The Torelli group is itself generated by elements admitting a simple description. A {\em bounding pair} is a set of curves $a, b \subset S$ that are disjoint and such that $S \setminus \{a,b\}$ is disconnected. A {\em bounding pair map} is the product $T_a T_b^{-1}$ of Dehn twists; it is straightforward to show that these are elements of the Torelli group. 

\begin{theorem}[Johnson, \cite{johnsonbps}]\label{theorem:johnsonBPs}
    For $g \ge 3$, $\cI(S)$ is generated by bounding pair maps. 
\end{theorem}

It is worth remarking that both $\cI(S)$ and $\cK(S)$ are in fact generated by a {\em finite} collection of such elements. This result for $\cI(S)$ is due to Johnson \cite{johnsonI}, while the result for $\cK(S)$ is a much more recent result of Ershov-He \cite{ershov-he}. Note however we will not use these finite generation results in our work.

\section{Bi-affine representations of the mapping class group}\label{S:biaffine}

Here we first define the notions of affine, co-affine, and bi-affine representations in general, then quickly specialize to the setting of the mapping class group, collecting some simple preliminary results. The main result of the section is \Cref{cor:biafftriv}, which gives a classification of bi-affine representations of $\Mod(S)$ with core $H$ (as always, $H = H_1(S;\C)$). 

\subsection{Bi-affine representations}

\begin{definition}[Affine, co-affine, bi-affine]\label{def:biaff}
    A group representation $\rho: G \to \GL(V)$ is {\em affine} if there is an invariant subspace $W \le V$ for which the quotient representation $\bar \rho: G \to \GL(V/W)$ is trivial. $\rho$ is {\em co-affine} if there is an invariant subspace $W$ for which $\rho|_W$ is trivial. $\rho$ is {\em bi-affine} if there are invariant subspaces $V_1 \le V_2 \le V$ for which $V_1$ and $V/V_2$ are both trivial. In this setting, the quotient $V_2/V_1$ is called the {\em core} of the representation.
\end{definition}

\begin{remark}\label{remark:duals}
    The dual of an affine representation is co-affine, and vice versa.
\end{remark}

\subsection{The unit tangent representation as an affine representation}\label{subsection:UT}
Here we explain how to understand $\tilde H$ and $\tilde H^*$ as a co-affine and affine representation, respectively. 

We first recall the construction of $\tilde H$ and its dual $\tilde H^*$. As above, $UTS$ denotes the unit tangent bundle of $S$, and set $\tilde H = H_1(UTS; \C)$. The projection $\pi: UTS \to S$ induces a surjection $\pi_*: \tilde H \to H$. The kernel is the trivial one-dimensional subrepresentation spanned by the class of the $S^1$ fiber. Thus there is a short exact sequence
\[
0 \to \C \to \tilde H \to H \to 0,
\]
realizing $\tilde H$ as a co-affine representation. Dually, $\tilde H^* = H^1(UTS;\C)$ is seen to be affine, with $H$ embedded as a submodule of codimension one via pullback.


\subsection{Extensions}
The next lemma shows that any extension of a bi-affine representation of $\Mod(S)$ is again bi-affine, so long as the dimension does not grow too much. 

\begin{lemma}\label{lemma:againDCS}
    Let $\rho: \Mod(S) \to \GL(V)$ be a bi-affine representation with respect to the filtration $V_1 \le V_2 \le V$. Suppose that $V$ embeds as a submodule of $\tilde V$, and that $\dim(\tilde V/V_2) = \dim(\tilde V/V) + \dim(V/V_2) \le 2g - 1$. Then $\tilde V$ likewise is bi-affine, with respect to the filtration $V_1 \le V_2 \le \tilde V$. 

    Dually, suppose $U \le \tilde V$ is a filtration of $\C[G]$-modules, with $V:= \tilde V/U$ bi-affine, filtered as $V_1 \le V_2 \le V$. Let $\tilde V_1 \le \tilde V_2 \le \tilde V$ be the preimage of this filtration in $\tilde V$. Suppose that $\dim(\tilde V_1) = \dim(U)+ \dim(V_1) \le 2g-1$. Then $\tilde V$ likewise is bi-affine, with respect to the filtration $\tilde V_1 \le \tilde V_2 \le \tilde V.$
\end{lemma}
\begin{proof}
    By hypothesis, $\dim(\tilde V/V_2) \le 2g-1$. By \Cref{thm:korkmaz2g-1}, $\tilde V/V_2$ must be trivial; as $V_1$ is trivial by hypothesis, this realizes $\tilde V$ as a bi-affine representation. The dual statement follows from this by \Cref{remark:duals}.
\end{proof}

\subsection{Classifying bi-affine representations of the mapping class group}\label{subsection:gpcohomlemma}

The objective of this subsection is \Cref{cor:biafftriv}, which shows that bi-affine representations of $\Mod(S)$ (with core $H$) are extremely simple. This is a homological calculation, built around a determination of certain $\Ext$ groups (recall that for $R$ any ring and $A,B$ any $R$-modules, $\Ext_R^1(A,B)$ classifies isomorphism classes of extensions $0 \to B \to E \to A \to 0$). We first recall some basic properties of $\Ext$ and its connection to group cohomology.

\begin{proposition}[cf. \cite{brown}, III.2.2]\label{prop:exttoH}
    Let $G$ be a group and let $M,N$ be $\C[G]$-modules. Then 
    \[
    \Ext^*_{\C[G]}(M,N) \cong H^*(G; \Hom_\C(M,N)).
    \]
\end{proposition}

In general, $\Ext^*_R(A,B)$ carries commuting actions of $\Aut_R(A)$ and $\Aut_R(B)$. When $A$ and/or $B$ is a trivial $R$-module, this structure is extremely simple.

\begin{lemma}\label{lemma:extGL}
    Let $M$ be a $\C[G]$-module, and let $V \cong \C^n$ be a {\em trivial} $\C[G]$-module.
    Then as a $\GL(V)$-module,
    \[
    \Ext^*_{\C[G]}(V,M) \cong V^* \otimes \Ext^*_{\C[G]}(\C,M) \qquad \mbox{and} \qquad \Ext^*_{\C[G]}(M,V) \cong \Ext^*_{\C[G]}(M, \C) \otimes V.
    \]
\end{lemma}
\begin{proof}
     $\Ext$ commutes with the formation of finite direct sums in either argument \cite[Proposition 3.3.4]{weibel}.
\end{proof}

For simplicity of notation, in the remainder of the section, the ring $R$ in $\Ext^*_R$ will always be $\C[\Mod(S)]$, and we will suppress it from the notation. Our determination of $\Ext$ groups in this setting is based on a fundamental result of Morita.

\begin{proposition}[Morita, cf. Propositions 4.1 and 6.4 of \cite{moritajacobians}]\label{lemma:h1modH}
    For $S = \Sigma_{g,1}$ or $\Sigma_{g,*}$,
    \[
    H^1(\Mod(S);H) \cong \C.
    \]
    For $S = \Sigma_g$,
    \[
    H^1(\Mod(S);H) = 0.
    \]
\end{proposition}

\begin{corollary}\label{cor:extcomps}
    Let $V,W$ be trivial $\Mod(S)$-modules. Then
    \[
    \Ext^1(V,W) = 0.
    \]
    If $S = \Sigma_{g,1}$ or $\Sigma_{g,*}$, then
    \[
    \Ext^1(V,H) \cong V^* \qquad \mbox{and} \qquad \Ext^1(H,W) \cong W,
    \]
    while if $S = \Sigma_g$, then
    \[
    \Ext^1(V,H) = \Ext^1(H,W) = 0.
    \]
\end{corollary}
\begin{proof}
    Using \Cref{prop:exttoH}, these translate into assertions about $H^1(\Mod(S);\Hom_\C(A,B))$ for various $A,B$. The first assertion follows from the fact that $H^1(\Mod(S);V) = 0$ for all $g \ge 0$ and all trivial $\Mod(S)$-modules $V$. The remaining assertions follow by combining \Cref{lemma:extGL} and \Cref{lemma:h1modH}.
\end{proof}

\begin{proposition}\label{prop:whencoboundary}
   Let $V_1$ and $V_3$ be trivial $\Mod(\Sigma_{g,1})$-modules, and let $[V_2] \in \Ext^1(H,V_1)$ be given. Suppose $[V_2] \ne 0$. Then $\Ext^1(V_3,V_2) = 0$.
\end{proposition}

\begin{proof}
   The short exact sequence $0 \to V_1 \to V_2 \to H \to 0$ induces the long exact sequence
\[
\Ext^1(V_3, V_1) \to \Ext^1(V_3,V_2) \to \Ext^1(V_3,H) \xrightarrow{\delta^1} \Ext^2(V_3,V_1).
\]
By \Cref{cor:extcomps} and \Cref{lemma:extGL}, this simplifies to
\[
0 \to \Ext^1(V_3,V_2) \to V_3^* \otimes \Ext^1(\C,H)  \xrightarrow{\delta^1} \Hom_\C(V_3, V_1) \otimes \Ext^2(\C,\C),
\]
and so $\Ext^1(V_3, V_2) = 0$ if and only if $\delta^1$ is injective. Moreover, it suffices to consider the case $V_1 = V_3 = \C$, and to take $V_2 = \tilde H$.

To proceed, we move to the setting of group cohomology, recasting $\delta^1: \Ext^1(\C,H) \to \Ext^2(\C,\C)$ as the connecting map $\delta^1: H^1(\Mod(\Sigma_{g,1}); H) \to H^2(\Mod(\Sigma_{g,1});\C)$ in the cohomology long exact sequence associated to the extension
\begin{equation}\label{eq:htilde}
    0 \to \C \to \tilde H \to H \to 0.
\end{equation}
 From here, we follow a calculation of Kawazumi-Souli\'e.\footnote{Their argument employs coefficients in $\Z$, not $\C$, but this is immaterial.} In \cite[Proposition 3.7]{KS}, they show that for all $g \ge 0$, the connecting map
 \[
 \delta^1: H^1(\Mod(\Sigma_{g,1}); H) \to H^2(\Mod(\Sigma_{g,1}); \C)
 \]
 is given by the (presently mysterious) formula 
 \[
 \delta^1(v) = \mu_*(m_{1,1}\cup v).
 \]
 The terms $m_{1,1}$ and $\mu_*$ require explanation. The class $m_{1,1} \in H^1(\Mod(\Sigma_{g,1}); H)$ is originally constructed in \cite{KM};\footnote{We follow Kawazumi-Souli\'e in referencing the unpublished expanded version of the original research announcement \cite{KMpub}.} as shown in \cite[Eq. (2.5)]{KS}, it coincides with the extension class of $\tilde H \in \Ext^1(H,\C) \cong H^1(\Mod(\Sigma_{g,1});H)$ for any $g \ge 2$. In particular, by \Cref{lemma:h1modH}, $m_{1,1}$ spans $H^1(\Mod(\Sigma_{g,1});H)$. The map $\mu_*: H \otimes H \to \C$ is the contraction map induced by the intersection pairing on $H$. Injectivity of $\delta^1$ now follows from the {\em contraction formula} established in \cite[Theorem 6.2]{KM}:
 \[
 \mu_*(m_{1,1} \cup m_{1,1}) = -e_1 \ne 0,
 \]
where $e_1 \in H^2(\Mod(\Sigma_{g,1});\C)$ is the first MMM class.
\end{proof}

\begin{corollary}\label{cor:biafftriv}
    Let $S$ be a surface of genus $g \ge 3$ which is either closed, has one boundary component, or one puncture. Let $V$ be a bi-affine representation of $\Mod(S)$ of arbitrary dimension with core $V_2/V_1 \cong H$. Then $V \cong W \oplus \C^n$, where $W$ is the symplectic representation $H$, the unit tangent representation $\tilde H$, or the dual $\tilde H^*$. If $S$ is closed, the latter two possibilities cannot arise.
\end{corollary}
\begin{proof}
Note first that any representation of $\Mod(\Sigma_{g,*})$ is {\em a fortiori} a representation of $\Mod(\Sigma_{g,1})$. Therefore it suffices to assume that $S$ is either closed or has one boundary component.

Let $V_1 \le V_2 \le V$ be the bi-affine filtration on $V$, determining associated classes $[V_2] \in \Ext^1(H,V_1)$ and $[V] \in \Ext^1(V/V_2,V_2)$. First suppose $[V_2] =0$, so that $V_2 \cong H \oplus V_1$. Then, by \Cref{lemma:extGL} and \Cref{cor:extcomps}, in the case $S = \Sigma_{g,1}$,
\[
\Ext^1(V/V_2,V_2) \cong \Ext^1(V/V_2, H) \oplus \Ext^1(V/V_2,V_1) \cong (V/V_2)^*,
\]
while if $S = \Sigma_g$, then by \Cref{cor:extcomps}, $\Ext^1(V/V_2,H) = 0$, so that $\Ext^1(V/V_2,V_2) = 0$.
As $\GL((V/V_2)^*)$ acts transitively on the nonzero elements of $(V/V_2)^*$, it follows that in the case of $S = \Sigma_{g,1}$ there is a {\em unique} isomorphism class of $\Mod(S)$-module forming a nontrivial extension of $V/V_2$ by $V_2 \cong H \oplus V_1$. 
This is represented explicitly by $\tilde H^* \oplus V_1 \oplus U$, where $U \le V/V_2$ is a subspace of codimension one, and is of the claimed form. If instead also $[V] = 0$ in $\Ext^1(V/V_1,V_1)$, then $V \cong H \oplus V_1 \oplus V/V_2$. This completes the proof in the case where $[V_2] = 0$.

Now suppose $[V_2] \ne 0$. By \Cref{prop:whencoboundary}, $\Ext^1(V/V_2,V_2) = 0$. The argument now proceeds as before: $V \cong V/V_2 \oplus V_2$, and as $\Ext^1(H,V_1) \cong V_1$ by \Cref{cor:extcomps}, we again find (by exploiting the action of $\GL(V_1)$ to construct an explicit representative) that $V_2 \cong \tilde H \oplus U$, where again $U \le V_1$ is codimension one. 
\end{proof}

In later arguments, we will make use of the following structural corollary of this result.
\begin{corollary}\label{lemma:DCScodim1}
Let $\rho: \Mod(S) \to \GL(V)$ be a bi-affine representation with core $H$. Let $c$ be a nonseparating simple closed curve, and $T_c$ the associated Dehn twist. Then $\rho(T_c)$ is unipotent and the $1$-eigenspace $E_1^c \le V$ is of codimension one.
\end{corollary}
\begin{proof}
By \Cref{cor:biafftriv}, $V = W \oplus \C^n$, where $W$ is one of $H, \tilde H$, or $\tilde H^*$. It therefore suffices to prove the claim for $W =H$ and $W = \tilde H^*$. If $W = H$ (and hence $\rho = \Psi$), then $\Psi(T_c)(x) = x + \pair{x,c}[c]$, which is unipotent and has $1$-eigenspace $[c]^\perp \le H$ of codimension one as claimed. 

In the case $W = \tilde H^*$, the action of a Dehn twist is given as
\[
\tilde \Psi^*(T_c) = \left( \begin{array}{c|c}\Psi(T_c) & \phi(T_c) \\ \hline 0 & 1 \end{array}\right),
\]
where $\phi: \Mod(S) \to H$ is a crossed homomorphism (obtained as a representative of the extension class for $[\tilde H^*] \in \Ext^1(\C,H) \cong H^1(\Mod(S);H)$). According to \cite[Theorem 4.2]{kasahara}, there is a constant $\lambda \in \C$ (depending on $\phi, c$, and a choice of orientation of $c$) such that $\phi(T_c) = \lambda[c]$. As we saw in the previous paragraph, $\Psi(T_c) - I$ likewise has image contained in the span of $[c]$, and annihilates $[c]$, so that $\tilde \Psi^*(T_c)$ has $1$-eigenspace of codimension one and is unipotent, as claimed.
\end{proof}

\section{A criterion to be bi-affine}

The goal of this section is to prove \Cref{prop:whenDCS}, which gives a sufficient condition under which a representation $\rho: \Mod(S) \to \GL(n,\C)$ is bi-affine with core $H$.
{\bf To streamline the discussion, we impose the standing convention that if left unspecified, the core of a  bi-affine representation is $\mathbf{H}$.}
We first recall an extremely useful result formulated by Korkmaz (readily derived from the vanishing of $H^1(\Mod(S);\C)$) which we will use throughout the rest of the paper, as well as two preparatory results about representations of the symplectic group.

\begin{lemma}[Flag triviality criterion (Korkmaz), Lemma 7.1 of \cite{korkmaz}]\label{lemma:flagtriv}
    Let $\rho: \Mod(S) \to \GL(n, \C)$ be a representation. If there is a flag
    \[
    0 = W_0 \subset W_1 \subset W_2 \subset \dots \subset W_k = \C^n
    \]
    of $\Mod(S)$-invariant subspaces such that each quotient $W_i/W_{i-1}$ is a trivial $\Mod(S)$-representation, then the image of $\rho$ is trivial. In particular this holds if $\dim(W_i/W_{i-1}) \le 2g -1$ for each $i = 1, 2, \dots, k$. 
\end{lemma}
Invariant flags turn out to be ubiquitous, in large part because of the following principle. Suppose $R \subset S$ is a subsurface, and let $c \subset S \setminus R$ be a simple closed curve. Then $\Mod(R)$ commutes with $T_c$, and since commuting linear transformations preserve (generalized) eigenspaces, any flag of generalized eigenspaces for $\rho(T_c)$ is a $\Mod(R)$-invariant flag. We will often use this principle without further comment.

    \begin{lemma}\label{lemma:smallestsp}
        For $g \ge 3$, the irreducible representations of $\Sp(2g,\C)$ of smallest dimension are the trivial representation $\C$ of dimension $1$, the standard representation $H$ of dimension $2g$, and $\wedge^2H/\C$ of dimension $\binom{2g}{2}-1$ (for $g = 3$ there is a second representation of dimension $\binom{2g}{2}-1 = 14$).
    \end{lemma}
    \begin{proof}
        The low-dimensional representations of semi-simple Lie groups are tabulated in \cite[Table 1]{AVE}. 
    \end{proof}
    
Recall that a {\em transvection} on $H$ is a transformation $T_x(y) = y + \pair{x,y}x$, where $x,y \in H$.
    \begin{lemma}\label{lemma:sprep}
        For $g \ge 3$, let $V$ be a $\Sp(2g,\Z)$-module of dimension at most $4g-1$. Suppose that every transvection acts unipotently on $V$. Then either $V$ is trivial, or else $V=H \oplus \C^d$; in particular, $V$ is bi-affine.
    \end{lemma}
    \begin{proof}
        A result of Harman \cite[Corollary 2.3]{harman} asserts that for $n > 2$, any representation $\rho: \SL_n(\Z) \to \GL_N(\C)$ for which the elementary matrices act unipotently extends to an algebraic representation $\tilde \rho: \SL_n(\C) \to \GL_N(\C)$. The proof of this uses only superrigidity for $\SL_n(\Z)$ and the congruence subgroup property for $\SL_n(\Z)$, both of which hold for $\Sp(2g,\Z)$ for $g \ge 2$. Thus $\rho$ extends to a representation $\rho: \Sp(2g,\C) \to GL(V)$. As $\Sp(2g,\C)$ is semi-simple, $V$ decomposes as a direct sum of irreducibles $V = \bigoplus_{i=1}^k V_i$. As $\dim(V) = 4g-1 < \binom{2g}{2}-1$, by \Cref{lemma:smallestsp}, at most one of the summands $V_i$ can be nontrivial, in which case it must be isomorphic to $H$.
    \end{proof}

    \begin{proposition}\label{prop:whenDCS}
    Let $g \ge 3$, and let $S = \Sigma_{g,1}$ be a surface of genus $g$ with one boundary component. Let $\rho: \Mod(S) \to \GL(n, \C)$ be a representation, endowing $V = \C^{n}$ with the structure of a $\Mod(S)$-module. Suppose that the following conditions hold:
    \begin{enumerate}[label=(\Alph*)]
        \item The restriction of $\rho$ to $\cK(S)$ is trivial,
        \item For any nonseparating curve $c \subset S$, $\rho(T_c)$ is unipotent. 
    \end{enumerate}
    Then for $n \le 4g-1$, either $V$ is trivial, or else $V$ is bi-affine (with core $H$).
\end{proposition}

\begin{proof}
    We proceed by induction on $m = n-2g$. Korkmaz's \Cref{thm:korkmaz2g} provides the base case $m = 0$. Before proceeding any further, observe that if $V$ satisfies hypotheses (A) and (B), then so does any submodule $W \le V$, and likewise for any quotient module $V/W$.

    We consider the restriction of $\rho$ to the Torelli group $\cI(S)$. By hypothesis (A), $\rho|_{\cI(S)}$ factors through $\cI(S)/ \cK(S)$. Recalling the discussion of \Cref{subsection:johnson}, $\rho|_{\cI(S)}$ factors through the Johnson homomorphism $\tau: \cI(S) \to A_S$; in particular, the image is abelian. By \Cref{theorem:johnsonBPs}, $\cI(S)$ is generated by bounding pair maps $T_a T_b^{-1}$, where $a$ and $b$ are disjoint and homologous. By hypothesis (B), it follows that every bounding pair map acts unipotently. In summary, the action of $\cI(S)$ on $V$ is both unipotent and abelian.

    In light of this, the fixed space $V^{\cI(S)}$ is nonempty. This is a $\Mod(S)$-submodule, and moreover since $\cI(S)$ acts on it trivially, the action of $\Mod(S)$ on $V^{\cI(S)}$ factors through the symplectic group $\Sp(2g,\Z)$. If $V^{\cI(S)} = V$, then the result follows by \Cref{lemma:sprep}. Otherwise, $V^{\cI(S)}$ is a proper submodule of $V$. By induction, either $V^{\cI(S)}$ is bi-affine, or else $V^{\cI(S)}$ is trivial. 
    
    In the former case, as $\dim(V) \le 4g-1$ and $V$ contains a copy of $H$ of dimension $2g$, the dimension bound of \Cref{lemma:againDCS} is seen to hold, and so by that result, $V$ is bi-affine. If $V^{\cI(S)}$ is trivial, then by \Cref{lemma:flagtriv}, either $V$ itself is trivial, or else the complement $V/V^{\cI(S)}$ must be a nontrivial $\Mod(S)$-module. By induction, this must be bi-affine, and again since $\dim(V) \le 4g-1$, the dimension bound of \Cref{lemma:againDCS} holds and we conclude that $V$ is bi-affine.
\end{proof}

\section{Unipotence}

In this section we establish two key results regarding the structure of Dehn twists in representations up to a certain range. In \Cref{prop:unipotence}, we show that Dehn twists act unipotently in every $\Mod(S)$-representation up to dimension $4g-3$. It is worth noting that a well-known folklore result shows that Dehn twists act {\em quasi-unipotently} (i.e. all eigenvalues are roots of unity) in {\em any} linear representation (see \cite[Corollary 3.5]{button} and/or \cite[Proposition 2.4]{ArSou}). The second key result is \Cref{prop:codimenison}, which shows that under stricter hypotheses, moreover the genuine $1$-eigenspace has codimension $1$.

\begin{proposition}\label{prop:unipotence}
    Let $g \ge 4$, and let $\rho: \Mod(S) \to \GL(n, \C)$ be a representation. Let $c \subset S$ be a nonseparating simple closed curve. Then for $0 \le n \le 4g-3$, $\rho(T_c)$ is unipotent.
\end{proposition}

\begin{remark}
    \Cref{prop:unipotence} holds more generally for $S = \Sigma_{g,n}^b$ a surface of genus $g \ge 4$ with any number of punctures and/or boundary components, since the results of Korkmaz on which it relies hold in this setting.
\end{remark}

The proof uses the following lemma, a generalization of \cite[Lemma 4.3]{korkmaz} with the same proof. Here and throughout, for a simple closed curve $c \subset S$, we write 
\[
E_{\lambda,k}^c = \ker((\rho(T_c)-\lambda I)^k)
\]
for the degree-$k$ generalized $\lambda$-eigenspace of $\rho(T_c)$. In the case $k = 1$ of the genuine eigenspace, we will write simply $E_\lambda^c$.

\begin{lemma}[Korkmaz, cf. Lemma 4.3 of \cite{korkmaz}]\label{lemma:generalizedinvariant}
    Let $b,c \subset S$ be nonseparating simple closed curves satisfying $i(b,c) = 1$. Let $\rho$ be a linear representation of $\Mod(S)$ and let $\lambda$ be an eigenvalue of $\rho(T_b)$ (and hence of $\rho(T_c)$ also). If $E_{\lambda,k}^b = E_{\lambda,k}^c$ for some $k \ge 1$, then $E_{\lambda,k}^b$ is $\Mod(S)$-invariant.
\end{lemma}

We now give the proof of \Cref{prop:unipotence}:

\begin{proof}
    We proceed by induction on $n$. The base cases $n\le 2g$ follow from Korkmaz (\Cref{thm:korkmaz2g}). For $n>2g$, suppose for contradiction that $\rho(T_c)$ has an eigenvalue $\lambda \ne 1$. Let $\lambda_\sharp$ denote the multiplicity of $\lambda$ as an eigenvalue. 

    Firstly, suppose $\lambda_\sharp = n$. If $\rho(T_c)=\lambda I_n$, applying $\rho$ to the lantern relation results in $\lambda^4=\lambda^3$ and hence $\lambda = 1$, a contradiction. Thus 
    \[
    1 \le \dim E_{\lambda}^c \le n-1 \le 4g-4.
    \]
    In particular, either $E_{\lambda}^c$ or $\C^n/ E_{\lambda}^c$ has dimension between $1$ and $2g-2$. Both of these are invariant under $\Mod(S\setminus\{c\})$, where $S\setminus\{c\}$ is a surface of genus $g-1$ obtained by cutting $c$ open in $S$. Since $g-1\ge 3$, any representation of $\Mod(S\setminus\{c\})$ with dimension between $1$ and $2g-2$ is known by Korkmaz (\Cref{thm:korkmaz2g}), and in particular Dehn twists act unipotently, so $\lambda=1$, a contradiction.

    Thus $\lambda_\sharp < n$. Let $E_{\lambda,gen}^c$ be the generalized $\lambda$-eigenspace. As above, either $E_{\lambda,gen}^c$ or $\C^n/E_{\lambda,gen}^c$ has dimension between $1$ and $2g-2$. If $1 \le \dim E_{\lambda,gen}^c \le 2g-2$, by a similar argument as above, we get $\lambda=1$, a contradiction. Thus $1 \le \dim \C^n/E_{\lambda,gen}^c \le 2g-2$, and $\Mod(S\setminus\{c\})$ acts on this space either trivially or via the symplectic representation; in either case, Dehn twists act unipotently.
    Therefore any nonseparating simple closed curve on $S\setminus\{c\}$ has the same generalized $\lambda$-eigenspace as $c$. In particular, taking two nonseparating simple closed curves $d,d'$ on $S\setminus\{c\}$ with $i(d,d')=1$, then by \Cref{lemma:generalizedinvariant}, $E_{\lambda,gen}^c$ is $\Mod(S)$-invariant, of dimension $<n$, so by induction, $\lambda=1$, concluding the proof.
\end{proof}

This also leads to the following corollary:
\begin{corollary}
    Any homomorphism from $\Mod(S)$ to a compact Lie group of dimension at most $4g-3$ is trivial.
\end{corollary}

\begin{proof}
   Given $f: \Mod(S) \to G$ such a homomorphism, the adjoint representation furnishes a linear representation of $\Mod(S)$ of dimension $\le 4g-3$, which sends Dehn twists to unipotent elements by \Cref{prop:unipotence}. But the only unipotent element in a compact Lie group is the identity matrix \cite{Borel1991}. Thus $\im(f)$ is contained in $Z(G)$, an abelian group, and since $g \ge 3$ by standing assumption, it follows that $\im(f)$ is trivial.
\end{proof}

For the results of the next section, it is also necessary to specify the dimension of the 1-eigenspace. We make use of the following piece of linear algebra that governs the structure of a generalized eigenspace.
\begin{lemma}[Jordan inequalities, cf. Section 2.2 of \cite{TSS}]\label{lemma:jordan}
Let $A \in \End(\C^{n})$ be a linear transformation, and let $\lambda$ be an eigenvalue of $A$. Consider the filtration
\[
0  = E_{\lambda,0} \le E_{\lambda,1} \le E_{\lambda,2} \le \dots \le E_{\lambda,d} = E_{\lambda,gen}
\]
of the generalized eigenspace $E_{\lambda,gen}$. Then the dimensions of the associated graded quotients form a non-increasing sequence:
\[
\dim(E_{\lambda,j}/E_{\lambda,{j-1}}) \ge \dim(E_{\lambda,{j+1}}/E_{\lambda,{j}})
\]
for $1 \le j \le d-1$.
\end{lemma}

\begin{proposition}\label{prop:codimenison}
   Let $\rho: \Mod_{g,1} \to \GL(2g+m, \C)$ be a representation with $0<m\le g-3$. Let $c \subset S$ be a nonseparating simple closed curve. Suppose that for $h = g$ or $g-1$, every nontrivial representation $\rho': \Mod_{h,1} \to \GL(2h+(m-1), \C)$ is bi-affine with core $H$. 
   Then the $1$-eigenspace $E_1^c$ of $\rho(T_c)$ has codimension at most 1. 
\end{proposition}
\begin{proof}
    By \Cref{prop:unipotence}, $T_c$ acts unipotently on $\C^{2g+m}$. We consider the flag of generalized eigenspaces
    \[
    0 \le E_{1,1}^c \le E_{1,2}^c \le \dots \le E_{1,d}^c = \C^{2g+m};
    \]
    in this notation, $E_{1,1}^c = E_1^c$ is the genuine eigenspace.
    This is invariant under the action of $\Mod(S\setminus\{c\})$. 
    By the Jordan inequalities (\Cref{lemma:jordan}), the sequence of dimensions $\dim(E_{1,j}^c/E_{1,j+1}^c)$ is nonincreasing, so that $\dim(E_{1}^c)$ is an upper bound on the dimension of any such quotient. By the flag triviality criterion (\Cref{lemma:flagtriv}), it follows that $\dim(E_1^c) \ge 2g-2$.

    Let $S'_c \subset S \setminus \{c\}$ be a subsurface homeomorphic to $\Sigma_{g-1,1}$. If $\dim(E_1^c) \le 2(g-1) + (m-1)$, then by hypothesis, the restriction of $\rho(\Mod(S'_c))$ to $E_1^c$ must be bi-affine or else trivial. By \Cref{lemma:againDCS}, it follows that $\rho(\Mod(S'_c))$ itself is either bi-affine or trivial, and then the claim follows from \Cref{lemma:DCScodim1}.

    Otherwise $\codim(E_1^c)\le 2$, so it remains only to rule out the case $\codim(E_1^c) = 2$. Let $b \subset S \setminus S_c'$ be a simple closed curve $b$ with $i(b,c) = 1$. Then by \Cref{lemma:generalizedinvariant}, either $E_1^c$ is invariant under $\rho(\Mod(S))$, or else $E_1^b \cap E_1^c$ is a $\Mod(S'_c)$-invariant subspace of dimension at most $2(g-1) + (m-1)$. In the former case, by hypothesis, the action of $\Mod(S)$ on $E_1^c$ is either bi-affine or trivial, and we conclude as in the preceding paragraph. In the latter case, we again conclude by hypothesis that the representation of $\Mod(S'_c)$ on $E_1^b \cap E_1^c$ is bi-affine or trivial, and the argument finishes along the same lines once again.
\end{proof}

\section{Transvective representations}\label{section:transvective}

The work of the previous section shows that in a certain range, Dehn twists act unipotently, and with genuine eigenspace of codimension $1$. Here we study representations of this type (which we call {\em transvective representations}) in greater detail. The culminating result is \Cref{prop:kgtriv}, which shows that transvective representations necessarily annihilate the Johnson kernel $\cK(S)$. 

\begin{definition}[Transvective representation]\label{def:transvective}
    A representation $\rho: \Mod(S) \to \GL(V)$ is {\em transvective} if for all nonseparating simple closed curves $c \subset S$, $\rho(T_c)$ is unipotent and $\codim(E_1^c) = 1$.
\end{definition}

It turns out that transvective representations have an extremely rigid structure on Dehn twists. In \Cref{lemma:transvection}-\ref{lemma:normalize}, we establish this theory.

\begin{lemma}\label{lemma:transvection}
    Let $\rho: \Mod(S) \to \GL(V)$ be a transvective representation. Then for any nonseparating simple closed curve $c \subset S$, there is $\alpha_c \in V^*$ and $v_c \in V$ for which
    \[
    \rho(T_c)(x) = x + \alpha_c(x)v_c
    \]
    for all $x \in V$, and $\alpha_c(v_c) = 0$
\end{lemma}

\begin{proof}
    Since $\codim(E_1^c) = 1$, there is $\alpha_c \in V^*$ and $v_c \in V$ for which 
    \[
    \rho(T_c) - I = \alpha_c v_c.
    \]
    Since $
    \rho(T_c)$ is unipotent, $\rho(T_c) - I$ is nilpotent, and so necessarily $\alpha_c(v_c) = 0$.
\end{proof}

\begin{lemma}\label{lemma:LI}
    Let $\rho: \Mod(S) \to \GL(V)$ be a transvective representation. Let $a, b \subset S$ be nonseparating simple closed curves. Suppose that either $i(a,b) = 1$, or else that $i(a,b) = 0$ but that $a$ and $b$ are not homologous. Let $\alpha_a, \alpha_b, v_a, v_b$ be as in \Cref{lemma:transvection}. Then $\alpha_a$ and $\alpha_b$ are linearly independent, and the same is true of $v_a$ and $v_b$.
\end{lemma}
\begin{proof}
     This is a variant of the method of proof of \Cref{lemma:generalizedinvariant}; we recall the argument. Suppose for contradiction that $\alpha_a, \alpha_b$ or $v_a, v_b$ are linearly dependent. For convenience, we will notate the span of $\alpha_a$ by $\C\alpha_a$, and similarly for other objects. Let $c \subset S$ be an additional simple closed curve satisfying the same topological constraints as $b$ (i.e. in the first case, $i(a,c) = 1$ and similarly in the second case). By the change-of-coordinates principle, there is $f \in \Mod(S)$ such that $f(a) = a$ and $f(b) = c$. Then $\rho(f)$ commutes with $\rho(T_a)$ and so preserves $\C\alpha_a$ and $\C v_a$. On the other hand, $\rho(f)$ conjugates $\rho(T_b)$ to $\rho(T_c)$, and so takes $\C\alpha_b$ to  $\C\alpha_c$, and $\C v_b$ to $\C v_c$. If $\C \alpha_a = \C \alpha_b$, this shows that likewise $\C \alpha_a = \C \alpha_c$, and similarly for $\C v_a, \C v_b, \C v_c$.

     Let $\cC$ be the graph with vertices in correspondence with isotopy classes of nonseparating simple closed curves on $S$, and with edges connecting $a,b$ if $i(a,b) = 1$; let $\cC'$ be the graph on the same vertex set, with $a$ joined to $b$ if $i(a,b) = 0$ and $a, b$ are non-homologous. Both $\cC$ and $\cC'$ are connected for $g \ge 3$ (see \cite[Chapter 4]{FM}). Thus if $\C \alpha_a = \C \alpha_b$ for some $a,b$ adjacent in $\cC$ or $\cC'$, and if $c$ is an arbitrary vertex of $\cC^{(')}$, there is a path $c_0 = a, c_1, \dots, c_k = c$ in $\cC^{(')}$. The argument of the above paragraph shows that $\C \alpha_a = \C \alpha_{c_1}$, and successively $\C \alpha_{c_i} = \C \alpha_{c_{i+1}}$, ultimately showing $\C \alpha_a = \C \alpha_c$. The same argument of course works with vectors in place of covectors. 

     We show that in either situation, this leads to a contradiction. If $\C \alpha_a = \C \alpha_b$ for all pairs of nonseparating simple closed curves on $S$, then this shows that the codimension-$1$ subspace $E_1^a = \ker(\alpha_a)$ is $\rho$-invariant and trivial; by the flag triviality criterion (\Cref{lemma:flagtriv}), $\rho$ itself is trivial. If $\C v_a = \C v_b$ for all pairs of curves $a,b$, we reduce to the previous situation by considering the dual representation. 
\end{proof}

\begin{lemma}\label{lemma:braid}
    Let $\rho: \Mod(S) \to \GL(V)$ be a transvective representation. Let $a, b \subset S$ be nonseparating simple closed curves satisfying $i(a,b) = 1$. Let $\alpha_a, \alpha_b, v_a, v_b$ be as in \Cref{lemma:transvection}. Then
    \[
    \alpha_a(v_b)\alpha_b(v_a) = -1.
    \]
\end{lemma}
\begin{proof}
    Since $i(a,b) = 1$, the twists $T_a, T_b$ satisfy the braid relation: $T_aT_bT_a = T_b T_a T_b$. A direct computation with the formula of \Cref{lemma:transvection} shows that
    \[
    \rho(T_aT_bT_a)(x) = x + (2 \alpha_a(x) + \alpha_b(x) \alpha_a(v_b) + \alpha_a(x) \alpha_b(v_a)\alpha_a(v_b)) v_a + (\alpha_b(x) + \alpha_a(x)\alpha_b(v_a))v_b
    \]
    and, by reversing the roles of $a$ and $b$,
    \[
    \rho(T_bT_aT_b)(x) = x + (2 \alpha_b(x) + \alpha_a(x) \alpha_b(v_a) + \alpha_b(x) \alpha_a(v_b)\alpha_b(v_a)) v_b + (\alpha_a(x) + \alpha_b(x)\alpha_a(v_b))v_a. 
    \]
    By \Cref{lemma:LI}, $v_a$ and $v_b$ are linearly independent. Comparing coefficients on $v_a$ (and using the fact that $\alpha_a$ is not identically zero, also by \Cref{lemma:LI}), we obtain the equation
    \[
    1 + \alpha_a(v_b)\alpha_b(v_a) = 0
    \]
    as desired.
\end{proof}

\begin{lemma}\label{lemma:commute}
    Let $\rho: \Mod(S) \to \GL(V)$ be a transvective representation. Let $a, c \subset S$ be nonseparating simple closed curves that satisfy $i(a,c) = 0$ but are not homologous. Let $\alpha_a, \alpha_c, v_a, v_c$ be as in \Cref{lemma:transvection}. Then
    \[
    \alpha_a(v_c) = \alpha_c(v_a) = 0.
    \]
\end{lemma}
\begin{proof}
    We compute
    \[
    \rho(T_c T_a) = x + \alpha_a(x) v_a + (\alpha_c(x) + \alpha_a(x) \alpha_c(v_a))v_c
    \]
    and
    \[
    \rho(T_a T_c) = x + \alpha_c(x) v_c + (\alpha_a(x) + \alpha_c(x) \alpha_a(v_c))v_a.
    \]
    By \Cref{lemma:LI}, $v_a,v_c$ are linearly independent. Comparing coefficients on $v_a$ and $v_c$ then yields the desired identities.
\end{proof}

Recall that a {\em chain} $a_1, \dots, a_k$ of simple closed curves in $S$ is a set of curves for which $i(a_i,a_{i+1}) =1$ for $i =1, \dots, k-1$ and for which $i(a_i,a_j) = 0$ for $\abs{i-j} > 1$. We say that a chain is {\em standardly embedded} if the classes $[a_1], \dots, [a_k]$ are linearly independent in $H_1(S;\Z)$. 

\begin{lemma}\label{lemma:normalize}
    Let $\rho: \Mod(S) \to \GL(V)$ be a transvective representation. Let $a_1, \dots, a_{2k}$ be a standardly-embedded chain; let $v_{a_1}, \dots, v_{a_{2k}}$ be the vectors associated to $\rho(T_{a_1}), \dots, \rho(T_{a_{2k}})$ as in \Cref{lemma:transvection}. Then $[v_{a_1}], \dots, [v_{a_{2k}}]$ are linearly independent in $V$, and can be normalized so that $\alpha_{a_i}(v_{a_{i+1}}) = - \alpha_{a_{i+1}}(v_{a_i}) = 1$ for $i = 1, \dots, 2k-1$, and $\alpha_{a_i}(v_j) = 0$ for $\abs{i-j} > 1$.
 \end{lemma}
 \begin{proof}
     To see that $\{[v_{a_i}]\}$ is linearly independent, suppose
     \[
     w = c_1 v_{a_1} + \dots + c_{2k} v_{a_{2k}} = 0.
     \]
     Applying $T_{a_1}$ and applying \Cref{lemma:braid} and \Cref{lemma:commute},
     \[
     0=T_{a_1}(w) = w + \alpha_{a_1}(w)v_a = c_2\alpha_{a_1}(v_{a_2}) v_a,
     \]
     showing that $c_2 = 0$. Successively applying $T_{a_3}, \dots, T_{a_{2k-1}}$ then shows that all coefficients $c_{2i} = 0$. Working from the other end, applying $T_{a_{2k}}$ shows that $c_{2k} = 0$; working backwards applying $T_{a_{2k-2}}, \dots, T_{a_{2}}$ then shows that all remaining coefficients $c_{2i+1} = 0$ as well. 

     The desired normalization can be defined by setting
     \[
     c_i = \prod_{j = 1}^{i-1} \alpha_{a_j}(v_{j+1})
     \]
     and then defining
     \[
     \alpha'_{a_i} = c_i \alpha_{a_i} \qquad \mbox{and} \qquad v_{a_i}' = \frac{v_{a_i}}{c_i}.
     \]
     It is then a routine calculation to verify that $\alpha'_i(v'_{a_{i+1}}) = 1$ and the other claimed relations.
 \end{proof}
    
We come to the main result of the section.

\begin{proposition}\label{prop:kgtriv}
   Let $\rho: \Mod(S) \to \GL(V)$ be a transvective representation. Then the restriction of $\rho$ to $\cK(S)$ is trivial.
\end{proposition}
\begin{proof}
    Let $a,b,c,d$ be a standardly-embedded chain of length $4$. Following \Cref{lemma:normalize}, we can extend $v_a, v_b, v_c, v_d$ to a basis of $V$. In such a basis, after normalizing {\em \`a la} \Cref{lemma:normalize}, the elements $\rho(T_a), \dots, \rho(T_d)$ have the following matrix expressions:
\begin{align*}
    \rho(T_a) = \left( \begin{array}{cccc|c}
                    1 & 1 & 0 & 0 & \alpha_a\\
                    0 & 1 & 0 & 0 & 0\\
                    0 & 0 & 1 & 0 & 0\\
                    0 & 0 & 0 & 1 & 0\\ \hline
                    0 & 0 & 0 & 0 & I
                \end{array} \right) &\qquad
                \rho(T_b) = \left( \begin{array}{cccc|c}
                    1 & 0 & 0 & 0 & 0\\
                    -1 & 1 & 1 & 0 & \alpha_b\\
                    0 & 0 & 1 & 0 & 0\\
                    0 & 0 & 0 & 1 & 0\\ \hline
                    0 & 0 & 0 & 0 & I
                \end{array} \right)\\
                \rho(T_c) = \left( \begin{array}{cccc|c}
                    1 & 0 & 0 & 0 & 0\\
                    0 & 1 & 0 & 0 & 0\\
                    0 & -1 & 1 & 1 & \alpha_c\\
                    0 & 0 & 0 & 1 & 0\\ \hline
                    0 & 0 & 0 & 0 & I
                \end{array} \right) &\qquad
                \rho(T_d) = \left( \begin{array}{cccc|c}
                    1 & 0 & 0 & 0 & 0\\
                    0 & 1 & 0 & 0 & 0\\
                    0 & 0 & 1 & 0 & 0\\
                    0 & 0 & -1 & 1 & \alpha_d\\ \hline
                    0 & 0 & 0 & 0 & I
                \end{array} \right).
\end{align*}
Here, it should be emphasized that the last column is actually a block, recording what happens to basis vectors beyond $v_a, \dots, v_d$. The entries $\alpha_a, \dots, \alpha_d$ are in fact covectors (row vectors), expressing these functionals in the chosen coordinates.
    
    By \Cref{theorem:johnsonII}, $\cK(S)$ is generated by separating twists of genus $1$ and $2$, and so it suffices to examine these two mapping classes. Let $a,b,c,d \subset S$ be a standardly-embedded chain. Then a regular neighborhood of $a \cup b$ is a subsurface of genus $1$, bounded by a separating curve $\Delta_1$. There is an alternate form of the chain relation \cite[Section 4.4.1]{FM} which asserts
    \[
    T_{\Delta_1} = (T_a^2 T_b)^4.
    \]
    A computation shows that
    \begin{equation}\label{eqn:aab}
        \rho(T_a^2T_b) = \left( \begin{array}{cccc|c}
                    -1 & 2 & 0 & 0 & 2(\alpha_a + \alpha_b )\\
                    -1 & 1 & 1 & 0 & \alpha_b \\
                    0 & 0 & 1 & 0 & 0\\
                    0 & 0 & 0 & 1 & 0\\ \hline
                    0 & 0 & 0 & 0 & I
                \end{array} \right)
    \end{equation}
    and hence
   \[
              \rho(T_a^2T_b)^2 = \left( \begin{array}{cccc|c}
                    -1 & 0 & 2 & 0 & 2\alpha_b \\
                    0 & -1 & 2 & 0 & -2\alpha_a \\
                    0 & 0 & 1 & 0 & 0\\
                    0 & 0 & 0 & 1 & 0\\ \hline
                    0 & 0 & 0 & 0 & I
                \end{array} \right);
    \]
    a final squaring shows that $\rho(T_{\Delta_1}) = \rho(T_a^2 T_b)^4 = I$ as required.

    The computation for a separating twist of genus $2$ proceeds along the same lines. We use the alternate formulation of the chain relation
    \[
    T_{\Delta_2} = (T_a^2 T_b T_c T_d)^8,
    \]
    where $\Delta_2$ is the boundary of the surface of genus $2$ containing $a,b,c,d$. Picking up from \eqref{eqn:aab},
    \[
    \rho(T_a^2 T_b T_c T_d) = \left( \begin{array}{cccc|c}
                    -1 & 0 & 0 & 2 & 2(\alpha_a  + \alpha_b + \alpha_c  + \alpha_d )\\
                    -1 & 0 & 0 & 1 & \alpha_b  + \alpha_c  + \alpha_d \\
                    0 & -1 & 0 & 1 & \alpha_c  + \alpha_d \\
                    0 & 0 & -1 & 1 & \alpha_d \\ \hline
                    0 & 0 & 0 & 0 & I
                \end{array} \right).
    \]
    Taking successive powers yields
    \[
    \rho(T_a^2 T_b T_c T_d)^2 = \left( \begin{array}{cccc|c}
                    1 & 0 & 2 & 0 & 2\alpha_d \\
                    1 & 0 & -1 & -1 & -2 \alpha_a  -\alpha_b  - \alpha_c \\
                    1 & 0 & -1 & 0 & -\alpha_b  + \alpha_d \\
                    0 & 1 & -1 & 0 & -\alpha_c  + \alpha_d \\ \hline
                    0 & 0 & 0 & 0 & I
                \end{array} \right)
    \]
    and
    \[
    \rho(T_a^2 T_b T_c T_d)^4 = \left( \begin{array}{cccc|c}
                    -1 & 0 & 0 & 0 & 2(\alpha_b + \alpha_d )\\
                    0 & -1 & 0 & 0 & -2 \alpha_a \\
                    0 & 0 & -1 & 0 & 2\alpha_d \\
                    0 & 0 & 0 & -1 & -2(\alpha_a  + \alpha_c )\\ \hline
                    0 & 0 & 0 & 0 & I
                \end{array} \right),
    \]
    and a final squaring yields $\rho(T_a^2T_bT_cT_d)^8 = I$.
\end{proof}

\section{Proof of \Cref{theorem:main}}

\begin{proof}[Proof of \Cref{theorem:main}]
    We will show that any representation $\rho: \Mod(S) \to \GL(2g+n, \C)$ for $g \ge 3$ and $0 \le n \le g-3$ is either trivial or else bi-affine with core $H$. The classification given in \Cref{theorem:main} then follows by \Cref{cor:biafftriv}.

    This will be established by induction on the pairs $(g,n)$ (again, with $g \ge 3$ and $0 \le n \le 3g-3$), endowed with the lexicographic ordering. The base case $(3,0)$ follows from Korkmaz (\Cref{thm:korkmaz2g}). Now let $(g,n)$ be given (in particular with $g \ge 4$), and assume that the result holds for any pair $(g',n')$ with $g' <g$ or else with $g' = g$ and $n' < n$. By \Cref{prop:unipotence}, $\rho(T_c)$ is unipotent for any nonseparating simple closed curve $c \subset S$. Invoking the inductive hypothesis, the hypotheses of \Cref{prop:codimenison} are satisfied, and so either $\rho$ is trivial, or else the $1$-eigenspace $E_1^c$ of $\rho(T_c)$ has codimension $1$. In this case, $\rho$ is transvective (\Cref{def:transvective}), and so by \Cref{prop:kgtriv}, the restriction of $\rho$ to $\cK(S)$ is trivial. Then by \Cref{prop:whenDCS}, it follows that $\rho$ is bi-affine with core $H$.
\end{proof}

\bibliography{references}{}
\bibliographystyle{alpha}
\end{document}